\documentclass[11pt,reqno]{amsart}

\usepackage[margin=1in]{geometry}
\usepackage{amsmath,graphicx,amssymb,amsthm,fixltx2e,enumerate,marvosym,amsfonts,graphics,cite}

\newtheorem{theorem}{Theorem}[section]
\newtheorem{lemma}[theorem]{Lemma} 
 
\newtheorem{corollary}[theorem]{Corollary}
\newtheorem{conjecture}[theorem]{Conjecture} 
 
\newtheorem{remark}[theorem]{Remark}

\newtheoremstyle{backref}
  {3pt}{3pt}		
  {\itshape}		
  {0pt}{\bfseries}	
  {.}				
  { }				
  {\thmname{#1} \thmnote{#3}}  

\theoremstyle{backref}
\newtheorem*{backthm}{Theorem}

\newtheorem*{backcor}{Corollary}

\newcommand{\abs}[1]{\left|{#1}\right|} 
\newcommand{\brac}[1]{\left[{#1}\right]} 
 
\newcommand{\setof}[2]{\left\{{#1}\,:\,{#2}\right\}}
\newcommand{\pr}[1]{\left({#1}\right)}

\newcommand{\cl}[1]{\overline{{#1}}}

\renewcommand{\t}[1]{\textnormal{{#1}}}

\newcommand{\dis}{\displaystyle}

\newcommand{\sub}{\subseteq}

\newcommand{\N}{\mathbb{N}}

\newcommand{\Z}{\mathbb{Z}}

\renewcommand{\d}{\delta}

\newcommand{\np}{\newpage}

\newcommand{\ld}{\ldots}

\newcommand{\K}{K_{\d,n-\d}}

\newcommand{\iso}{\cong}

\newcommand{\diam}{\textnormal{diam}}
\newcommand{\girth}{\textnormal{girth}}
\newcommand{\de}{\delta}
\newcommand{\D}{\Delta}

\title{A new method for enumerating independent sets of a fixed size in general graphs}
\author[J.~Alexander]{James Alexander}
\address[J.~Alexander]{University of Delaware, Newark DE 19716}
\email[J.~Alexander]{alex@math.udel.edu}
\author[T.~Mink]{Tim Mink}
\address[T.~Mink]{Montclair State University, Montclair NJ 07043}
\email[T.~Mink]{minkt@mail.montclair.edu}

\begin{document}
\maketitle

\begin{abstract}
	We develop a new method for enumerating independent sets of a fixed size in general graphs, and we use this method to show that a conjecture of  Engbers and Galvin holds for all but finitely many graphs. We also use our method to prove special cases of a conjecture of Kahn. In addition, we show that our method is particularly useful for computing the number of independent sets of small sizes in general regular graphs and Moore graphs, and we argue that it can be used in many other cases when dealing with graphs that have numerous structural restrictions.
\end{abstract}

\section{Introduction and notation} 
\label{sec:introduction}

We consider only simple graphs. Let $i_t(G)$ denote the number of independent sets in a graph $G$ of size $t$, and $i(G):=\sum_{t=0}^\infty i_t(G)$. The enumeration of independent sets in graphs with various degree restrictions has been a topic of much interest lately. Kahn made a breakthrough in these problems when he proved the following result in \cite{KAHN}. 

\begin{theorem}[Kahn]\label{thm:kahn}
	For any $n,r\in\N$, if $G$ is an $n$-vertex, $r$-regular bipartite graph with $r\geq 1$, then
\[
		i(G)\leq i(K_{r,r})^{\frac{n}{2r}}.
\]
\end{theorem}

In \cite{ZHAO}, Zhao extended this result to all regular graphs as follows:

\begin{theorem}[Zhao]\label{thm:zzhow}
	For any $n,r\in\N$, if $G$ is an $n$-vertex, $r$-regular graph with $r\geq 1$, then
\[
		i(G)\leq i(K_{r,r})^{\frac{n}{2r}}.
\]
\end{theorem}

In \cite{GALV}, Galvin relaxed the maximum degree restriction and conjectured the following, which he proved asymptotically. 

\begin{conjecture}[Galvin]\label{conj:galvin}
	For any $n,\d\in\N$, if $G$ is an $n$-vertex graph with $\d(G)\geq\d$, such that $n\geq 2\de$, then
	\[
		i(G)\leq i(\K).
	\]
\end{conjecture}

Note the assumption that $n\geq 2\d$ is only to ensure that the proposed extremal graph satisfies the minimum degree restriction. This was proven for bipartite graphs, and in other special cases, by Alexander, Cutler, and Mink in \cite{BIPO}. This conjecture was recently resolved when  the following stronger version was proven by Cutler and Radcliffe in \cite{CRADFT}. 

\begin{theorem}[Cutler, Radcliffe]\label{thm:cutrad}
For any $n,\d\in\N$, write $n=a(n-\d)+b$ with $0\leq b<n-\d$. If $G$ is an $n$-vertex graph with $\d(G)\geq\d$, 
\[
i(G)\leq a(2^{n-\d}-1)+2^b. 
\]
\end{theorem}
We notice that Conjecture~\ref{conj:galvin} does, in fact, follow from this, for when $n\geq 2\d$, Theorem~\ref{thm:cutrad} states that $i(G)\leq 2^{n-\d}+2^\d-1=i(\K)$. In \cite{GEN}, Engbers and Galvin proposed a generalization of Conjecture~\ref{conj:galvin} to independent sets of fixed sizes, which was the following. 

\begin{conjecture}[Engbers, Galvin]\label{conj:gen}
	For any $n,\d\in\N$, if $G$ is an $n$-vertex graph with $\d(G)\geq\d$, such that $n\geq 2\de$, then for all $t\geq 3$,
	\[
		i_t(G)\leq i_t(\K).
	\]
\end{conjecture}

Though the $t\geq 3$ condition may seem a bit unnatural, one may note that the conjecture, though true for $t\leq1$, fails for $t=2$. The smallest counterexample can be found when $n=4$ and $\d=1$, in which case $i_2(K_{1,3})< i_2(2K_{1,1})$ can easily be checked. Conjecture~\ref{conj:gen} was proven for bipartite graphs, and in the case that $n=2\d$ by Alexander, Cutler, and Mink in \cite{BIPO}. It was also proven for the cases $\d\leq 3$, and the cases $t\geq 2\d+1$ in \cite{GEN}. One of our main results is the following, Conjecture~\ref{conj:gen} when $n\geq\frac{(\d+1)(\d+2)}3$, the proof of which can be found in Section~\ref{sec:indmindeg}. Before stating this theorem, we note that progress on this conjecture was also made in \cite{MCLAW} by Law and McDiarmid using different methods, and that our result can be viewed as a generalization of this work done in \cite{MCLAW}. We also mention that this conjecture was resolved in \cite{POSHIN}, a manuscript by Gan, Loh, and Sudakov, after the submission of this paper. 

\begin{theorem}\label{thm:mainresultofmindeg}
	For any $n,\d\in\N$, if $G$ is an $n$-vertex graph with $\d(G)\geq\d$, and if $n\geq\frac{(\d+1)(\d+2)}{3}$, then for all $t\geq3$, 
	\[
		i_t(G)\leq i_t(\K).
	\]
\end{theorem}

Next, we focus on a generalization of Theorem~\ref{thm:zzhow}. A number of different generalizations of  this theorem, as well as the weaker Theorem~\ref{thm:kahn}, have been proposed, some of which can be found in \cite{KAHN}, \cite{KAHNNEWK} and \cite{ZHAO}. Of the proposed generalizations of Theorem~\ref{thm:zzhow}, there has been one made, in \cite{KAHN}, to independent sets of fixed sizes (a generalization of Theorem~\ref{thm:zzhow} that can be thought of as being analogous to the Conjecture~\ref{conj:gen} generalization of Conjecture~\ref{conj:galvin}).  This conjecture, which we state in different language using the notation of \cite{FUNC}, is Conjecture~\ref{conj:our}. For a given power series $P(x)$, let $[x^t]P(x)$ denote the coefficient of $x^t$ in $P(x)$. 

\begin{conjecture}[Kahn] \label{conj:our}
For any $n,r\in\N$, if $G$ is an $n$-vertex, $r$-regular graph with $r\geq1$, then for all $0\leq t\leq n/2$,
\begin{equation}
i_t(G)\leq[x^t]\pr{2(1+x)^{r}-1}^{\frac{n}{2r}}, \label{eq:generalizedkahn}
\end{equation}
with equality holding if $2r$ divides $n$ and $G=\frac{n}{2r}K_{r,r}$. 
\end{conjecture}

In Section~\ref{sec:bipartite}, we confirm that this is, in fact, a generalization of Theorem~\ref{thm:zzhow} (and therefore of Theorem~\ref{thm:kahn}),  i.e., that if Conjecture~\ref{conj:our} holds, then so does Theorem~\ref{thm:zzhow}, which was not previously shown. In particular, we prove the following. 

\begin{theorem}\label{justificationofgen}
If Conjecture~\ref{conj:our} holds, then so must Theorem~\ref{thm:zzhow}. 
\end{theorem}

We note that our proof of this theorem will be a short, combinatorial argument, whereas the only known proof of Theorem~\ref{thm:zzhow} relies on Theorem~\ref{thm:kahn}, the proof of which uses probabilistic methods and requires additional development. We also prove the following, which verifies the conjecture for all $t\leq3$, and for $t=4$ when $G$ is assumed triangle-free. We note that essentially no cases of this conjecture have been proven previously. 

\begin{theorem}\label{thm:thelast}
For any $n,r\in\N$, if $G$ is an $n$-vertex, $r$-regular graph with $r\geq1$, then $G$ satisfies \t{(\ref{eq:generalizedkahn})} for all $t\leq3$, and for $t=4$ when $G$ is assumed triangle-free, with equality holding if and only if $2r$ divides $n$ and $G=\frac{n}{2r}K_{r,r}$. 
\end{theorem}

All proofs of these results, and of additional results presented later, rely on new methods that we develop for independent set enumeration in Section~\ref{sec:proof}, which were inspired by Whitney's Broken Circuit theorem, which was first stated and proven in \cite{WHIT}. These methods differ very much from those previously used to study the independent set structure of a graph, but other methods inspired by Whitney's broken circuit theorem have been used to count graph colorings and graph homomorphisms (which are related structures) in \cite{SOLGAV}. For these new methods and their proofs, we must develop some notation. For this, we mainly follow \cite{WEST}. 

For a graph $G$ and $v\in V(G)$, let $N_G[v]:=N_G(v)\cup\{v\}$. For any $J\sub V(G)$ and $S\sub E(G)$, let $G[J]$ denote the subgraph of $G$ with vertex set $J$ and edge set $\{vw\in E(G):v,w\in J\}$, the so-called \emph{subgraph of $G$ induced by $J$}, and let $G[S]$ denote the subgraph of $G$ with edge set $S$, and vertex set consisting of all vertices which are incident to edges in $S$. For any graph $H$, let 
\begin{equation}
H(G):=\{J\sub V(G):G[J]\cong H\},\label{setsofsubgraphsofgraph}
\end{equation}
e.g., $C_5(G)$ be the set of $J\sub V(G)$ with $G[J]$ a pentagon. Let lowercase letters denote the sizes of these sets; for example, let $c_5(G)$ be defined as $|C_5(G)|$, the number of subsets $J\sub V(G)$ such that $G[J]$ is a pentagon, and let $k_3(G)$ be defined as $|K_3(G)|$, the number of $J\sub V(G)$ such that $G[J]$ is a triangle. As is common, for each $t\in\Z_{\geq0}$, let $I_t$ denote the graph isomorphic to $\cl{K_t}$, and let $P_t$ denote the path on $t$ vertices. 

Let $G\uplus H$ denote the graph consisting of two components, $G_1$ and $G_2$, with $G_1\iso G$ and $G_2\iso H$. We note that it is standard to denote this by $G+H$ rather than $G\uplus H$, but that we avoid this, as it causes ambiguity when using lowercase letters to denote the sizes of these sets (see, for example, the statement of Corollary~\ref{cor:ithree}(\ref{cor:ithree:b}), which would not have clear meaning with the traditional notation). Let $H\leq G$ mean that $H$ is a subgraph of $G$. For any $J\sub V(G)$, we call a set of edges $S\sub E(G[J])$ an \emph{edge-covering} of $G[J]$ if every vertex of $V(G[J])$ is incident to at least one edge of $S$. 

Our new methods center around a formula which we develop, which enumerates the number of independent sets of a fixed size in a general graph in terms of the number of edge-coverings of certain subgraphs. For this, we define function $\psi:V(G)\rightarrow\Z_{\geq0}$ as follows. For every $J\sub V(G)$, let $\psi_e(J)$ denote the number of edge-coverings of $G[J]$ by an even number of edges of $E(G[J])$, $\psi_o(J)$ denote the number of edge-coverings of $G[J]$ by an odd number of edges of $E(G[J])$, and $\psi(J):=\psi_e(J)-\psi_o(J)$. The following observation will be helpful. 

\begin{remark}\label{rem}
Observe that for any graph $G$, if $G[J]$ contains any isolated vertices, then $\psi(J)=0$, but the converse need not hold \t{(}one can easily check, for example, that if $G[J]\iso P_4$, then $\psi(J)=0$\t{)}.
\end{remark}

For ease of notation, for any $j\in\Z_{\geq0}$ and any set $A$, let $A^{(j)}:=\{B\sub A:|B|=j\}$, as opposed to the more common $\binom{A}{j}$, as in \cite{BELACOMBO}. We have now developed the necessary notation for our new methods, which are the topic of Section~\ref{sec:proof}. Once these methods are developed, we use them to prove Theorem~\ref{thm:mainresultofmindeg} and some corollaries in Section~\ref{sec:indmindeg}, then to discuss Conjecture~\ref{conj:our} and prove our results related this conjecture in Section~\ref{sec:bipartite}, after which we use our methods to study the independent set structure of certain regular graphs in Section~\ref{sec:regular}. 

\section{Our main result and immediate consequences} 
\label{sec:proof}

In \cite{WHIT}, Whitney used a combinatorial sieve to develop a formula for counting the number of colorings of a general graph $G$ of a fixed size in terms of broken circuits, this being his well-known Broken Circuit theorem. This provided a method for counting the number of ways that a graph can be partitioned into a fixed number of independent sets, a global property about the independent set structure of the graph, in terms of the structure of certain subgraphs, a more convenient local property. Inspired by his methods, we use a similar sieve method to develop a formula which enumerates the number of independent sets of a fixed size in a general graph, an oftentimes difficult to bound global property, in terms of the number of vertex covers of certain subgraphs. This formula is the following.

\begin{theorem}\label{lem:main}
For any $n\in\N$, any $n$-vertex graph $G$, and for any nonnegative integer $t$,
\begin{equation}
i_t(G)=\sum_{j=0}^t\binom{n-j}{t-j}N_j(G),\label{newform}
\end{equation}
where $\dis N_j(G):=\sum_{J\in V^{(j)}}\psi(J)$ for $j\geq1$, and $N_0(G):=1$. 
\end{theorem}

Although the language and motivations are different, it can be shown that this formula is equivalent to a result in \cite{PHYSICS}. With Remark~\ref{rem} in mind, one can quickly check that $N_1(G)=0$, $N_2(G)=-|E(G)|$ and $N_3(G)=p_3(G)+2\cdot k_3(G)$ for any graph $G$. Moreover, if we let $R_{4,1}$ denote the graph obtained by removing one edge from a copy of $K_4$, and $R_{4,2}$ denote the only graph not isomorphic to a quadrilateral that can obtained by removing two edges from a copy of a $K_4$, then the seven non-isomorphic graphs on four vertices without isolated vertices are $K_{1,3}$, $K_2\uplus K_2$, $P_4$, $C_4$, $R_{4,2}$, $R_{4,1}$ and $K_4$, and it is straightforward to check that if $J\in V^{(4)}$ is isomorphic to $K_{1,3}$, $K_2\uplus K_2$, $P_4$, $C_4$, $R_{4,2}$, $R_{4,1}$ or $K_4$, then $\psi(J)$ is equal to $-1$, $1$, $0$, $-1$, $-1$, $-2$ or $-3$, respectively. Thus, by Remark~\ref{rem}, 
\begin{equation}
N_4(G)=k_2\uplus k_2(G)-k_{1,3}(G)-r_{4,2}(G)-c_4(G)-2r_{4,1}(G)-3k_4(G). \label{eq:theonethatfixesit}
\end{equation}

\begin{proof}[Proof of Theorem~\ref{lem:main}]
Let $m:=|E(G)|$. For each $e\in E(G)$, let $A_e:=\{A\in V^{(t)}:e\sub A\}$. Then, $i_t(G)=\abs{\cap_{e\in E}\cl{A_e}}$, where the complement is taken within $V^{(t)}$. By inclusion-exclusion,
\begin{equation}
i_t(G)=\binom{n}t-\sum_{e\in E}\abs{A_e}+\sum_{\substack{\t{distinct}\\ e_1,e_2\in E}}\abs{A_{e_1}\cap A_{e_2}}-\cdots+(-1)^{m}\abs{A_{e_1}\cap A_{e_2}\cap\cdots\cap A_{e_m}} . \label{one}
\end{equation}
For each fixed $e_1,...,e_i\in E$, 
\begin{align}
\abs{A_{e_1}\cap\cdots\cap A_{e_i}}&=|\{A\in V^{(t)}:e_1\cup e_2\cup\cdots\cup e_i\sub A\}|\\
						&=\binom{n-|e_1\cup e_2\cdots\cup e_i|}{t-|e_1\cup e_2\cdots\cup e_i|}, \label{two}
\end{align}
since the other vertices (those not endpoints of any $e_k$ for $k\in[i]$) can be chosen freely. By (\ref{one}) and (\ref{two}), 
\begin{equation}
i_t(G)=\binom{n}t-\sum_{e\in E}\binom{n-|e|}{t-|e|}+\sum_{\substack{\t{distinct}\\ e_1,e_2\in E}}\binom{n-|e_1\cup e_2|}{t-|e_1\cup e_2|}-\cdots+(-1)^{m}\binom{n-|e_1\cup e_2\cup\cdots\cup e_m|}{t-|e_1\cup e_2\cup\cdots\cup e_m|} . \label{three}
\end{equation}
Notice that for each $i\in[m]$ and any $e_1,...,e_i\in E$, we have $1\leq|e_1\cup e_2\cdots\cup e_i|\leq 2i$; and for fixed $j\in[2i]$, that $|e_1\cup e_2\cup\cdots\cup e_i|=j$ if and only if $n(G[\{e_1,...,e_i\}])=j$. Thus, if we let 
\begin{equation}
n_{ij}:=|\{S\in E^{(i)}:n\left(G[S]\right)=j\}|, \label{four}
\end{equation}
the number of times that $(-1)^i\binom{n-j}{t-j}$ appears as a term in (\ref{three}) is $n_{ij}$, $i\in[m]$, $j\in[2i]$. Therefore,
\begin{equation}
i_t(G)=\binom{n}t+\sum_{i\in[m]}\sum_{j\in[2i]}n_{ij}(-1)^i\binom{n-j}{t-j}. 
\end{equation}
Since the $\binom{n-j}{t-j}=0$ if $j>t$, we may adjust the bounds of $j$ to be from $1$ to $t$ (with no harm if $j>2i$ since the terms will zero by definition). Thus, we see that
\begin{align}
i_t(G)&=\binom{n}t+\sum_{i=1}^m\sum_{j=1}^t n_{ij}(-1)^i\binom{n-j}{t-j}\\
	&=\binom{n}t+\sum_{j=1}^t\binom{n-j}{t-j}\sum_{i=1}^m n_{ij}(-1)^i.\label{referenceeyetee}
\end{align}
For fixed $j$, let us interpret $\sum_{i=1}^m n_{ij}(-1)^i$. Let $[m]_e$ denote the even elements of $[m]$ and $[m]_o$ be defined similarly. Then, 
\begin{equation}
\sum_{i=1}^m n_{ij}(-1)^i=\sum_{i\in[m]_e} n_{ij}-\sum_{i\in[m]_o} n_{ij}.\label{neggg}
\end{equation}
To better understand (\ref{neggg}), we turn our attention back to (\ref{four}). Notice that $n_{ij}$ can also be interpreted as the number of subgraphs with exactly $i$ edges and $j$ vertices, none of which are isolated. Thus, for fixed $j$, $n_{ij}$ is exactly the number of ways to choose $i$ edges and $j$ vertices so that every one of the $j$ vertices is incident to at least one of the $i$ edges. Thus, 
\begin{align}
\sum_{i\in[m]_e} n_{ij}&=\sum_{i\in[m]_e}\abs{\setof{(S,J)\in E^{(i)}\times V^{(j)}}{S\sub E(G[J])\t{ and $G[J]$ has no isolated vertices}}}\\
			&=\sum_{J\in V^{(j)}}\sum_{i\in[m]_e}\abs{\setof{S\in E(G[J])^{(i)}}{\t{ Graph $G[J]$ has no isolated vertices}}}\\
			&=\sum_{J\in V^{(j)}}\abs{\setof{S\sub E(G[J])}{\t{ $|S|\in2\Z$ and $S$ covers $J$}}}\\
			&=\sum_{J\in V^{(j)}}\psi_e(J),\label{now}
\end{align}
with the last equality holding by definition of $\psi_e(J)$. We can show similarly that 
\begin{equation}
\sum_{i\in[m]_o} n_{ij}=\sum_{J\in V^{(j)}}\psi_o(J).\label{later}
\end{equation}
Putting (\ref{now}) and (\ref{later}) into (\ref{neggg}), and recalling the definition of $\psi$, we see that 
\begin{equation}
\sum_{i=1}^m n_{ij}(-1)^i=\sum_{J\in V^{(j)}}\psi(J). 
\end{equation}
Applying this to (\ref{referenceeyetee}), we see that 
\begin{equation}
i_t(G)=\binom{n}t+\sum_{j=1}^t\binom{n-j}{t-j}\sum_{J\in V^{(j)}}\psi(J). 
\end{equation}
Thus, if we define $N_j(G)$ as we did for $j\geq 1$, and $N_0$ to be $1$, this is  the desired result. 
\end{proof}

We will begin to demonstrate the usefulness of Theorem~\ref{lem:main} by proving two very general corollaries which will be the integral tools in proving all results of the following sections. The first of these is a generalization of a classical result of Goodman, the main result of \cite{GOODMAN}. This corollary is the following. 

\begin{corollary}\label{cor:ithree}
Fix any $n,m\in\N$, and let $G$ be any $n$-vertex, $m$-edge graph. Taking all graph parameters to be functions of $G$ \t{(}e.g., $i_3=i_3(G)$ and $d(v)=d_G(v)$\t{)}, we have that 
\begin{enumerate}[(a)]
\item $i_3+k_3=\binom{n}{3}-(n-2)m+\sum_{v\in V}\binom{d(v)}{2}$; \label{cor:ithree:a}
\item $i_4-k_4=\binom{n}{4}-\binom{n-2}{2}m+(n-3)\brac{\sum_{v\in V}\binom{d(v)}{2}-k_3}-\sum_{v\in V}\binom{d(v)}{3}+k_2\uplus k_2-c_4$.\label{cor:ithree:b}
\end{enumerate}
\end{corollary}

We notice that these formulas in Corollary~\ref{cor:ithree} will simplify a great deal, without too much work, if we impose a regularity condition on $G$ (in particular, that the formulas will no longer contain sums, as $d(v)$ will be the same for every $v\in V$). Our second corollary formalizes this. 

\begin{corollary}\label{lem:reg}
Fix any $n,r\in\N$, and let $G$ be any $n$-vertex, $r$-regular graph. Taking all graph parameters to be functions of $G$,
\begin{enumerate}[(a)]
\item $i_3+k_3=\binom{n}{3}-(n-2)\frac{nr}{2}+n\binom{r}{2}$;\label{parta}
\item $i_4-k_4=\binom{n}{4}-\binom{n-2}{2}\frac{nr}2+(n-3)\brac{n\binom{r}{2}-k_3}-n\binom{r}{3}+k_2\uplus k_2-c_4$.\label{partb}
\end{enumerate}
In particular, all $n$-vertex, $r$-regular, triangle-free graphs have the same number of independent sets of size three, and the number of independent sets of size four in a triangle-free regular graph is a function of only the number of quadrilaterals and $K_2\uplus K_2$ subgraphs. 
\end{corollary}

This corollary, Corollary~\ref{lem:reg}, will prove very useful when studying regular graphs, such as when we are discussing the proposed generalization of the Theorem of Zhao, Theorem~\ref{thm:zzhow}, in Section~\ref{sec:bipartite}. Before we prove these two corollaries, we state and prove the following lemma, from which both corollaries will easily follow. 

\begin{lemma}\label{OKlemma}
Fix any $n,m\in\N$, and let $G$ be any $n$-vertex, $m$-edge graph. Taking all graph parameters to be functions of $G$,
\begin{enumerate}[(a)]
\item $N_1=0$;\label{ahha}
\item $N_2=-m$;\label{ahhb}
\item $N_3=\sum_{v\in V}\binom{d(v)}{2}-k_3$; \label{eq:nthree}
\item $N_4=-\sum_{v\in V}\binom{d(v)}{3}+k_2\uplus k_2-c_4+k_4$. \label{eq:nfour}
\end{enumerate}
\vspace{0.02in}
In particular, if $G$ is $r$-regular for some $r\in\N$, then we have that $N_2=-\frac{nr}{2}$, $N_3=n\binom{r}{2}-k_3$, and $N_4=-n\binom{r}{3}+k_2\uplus k_2-c_4+k_4$. 
\end{lemma}

\begin{proof}
Assume for the duration of this proof, as in the statement, that all graph parameters are taken to be functions of $G$. We proved (\ref{ahha}) and (\ref{ahhb}) in the remarks immediately following the statement of Theorem~\ref{lem:main}, so we start by proving (\ref{eq:nthree}). For each $v\in V$ and distinct $v_1,v_2\in N(v)$, we have $G[v,v_1,v_2]\iso P_3$ or $G[v,v_1,v_2]\iso K_3$. Conversely, looking at pairs of vertices in the neighborhood of each $v\in V$, each subgraph isomorphic to $P_3$ is induced this way uniquely (i.e., has one root) and each subgraph isomorphic to $K_3$ is induced this way three times. Thus, 
\begin{equation}
\sum_{v\in V}\binom{d(v)}{2}=p_3+3k_3. \label{eq:forithree}
\end{equation}
A combination of (\ref{eq:forithree}) and the remarks immediately following the statement of Theorem~\ref{lem:main} give us that 
\begin{equation}
N_3=\sum_{v\in V}\binom{d(v)}{2}-k_3.
\end{equation}
It remains to prove (\ref{eq:nfour}), which we do  in a similar way. For each $v\in V$ and distinct $v_1,v_2,v_3\in N(v)$, we have that $G[v,v_1,v_2,v_3]$ is isomorphic to either $K_{1,3}$, $R_{4,2}$, $R_{4,1}$ or $K_4$ ($R_{4,2}$ and $R_{4,1}$ as defined in the remarks immediately following the statement of Theorem~\ref{lem:main}). Conversely, looking at triples of vertices in the neighborhood of each $v\in V$, each subgraph isomorphic to $K_{1,3}$ is induced this way uniquely, as is each subgraph isomorphic to $R_{4,2}$. Each subgraph isomorphic to $R_{4,1}$ is induced twice this way, and each subgraph isomorphic to $K_4$ is induced this way four times. Therefore, 
\begin{equation}
\sum_{v\in V}\binom{d(v)}3=k_{1,3}+r_{4,2}+2r_{4,1}+4k_4. \label{eq:lastnfourpiece}
\end{equation}
A combination of (\ref{eq:theonethatfixesit}) and (\ref{eq:lastnfourpiece}) gives
\begin{equation}
N_4=-\sum_{v\in V}\binom{d(v)}{3}+k_2\uplus k_2-c_4+k_4. 
\end{equation}
\end{proof}

We now notice that Corollary~\ref{cor:ithree}(\ref{cor:ithree:a}) is exactly the statement obtained by letting $t=3$ in  Theorem~\ref{lem:main}, and applying Lemma~\ref{OKlemma} parts (\ref{ahha})-(\ref{eq:nthree}). Similarly, Corollary~\ref{cor:ithree}(\ref{cor:ithree:b}) is obtained the same way by letting $t=4$ and applying Lemma~\ref{OKlemma} parts (\ref{ahha})-(\ref{eq:nfour}). The statements of  Corollary~\ref{lem:reg} is then obtained the same way by applying the particular statement at the end of Lemma~\ref{OKlemma}.  

\begin{remark}
It is worth noting that by replacing all occurrences of edges by independent sets of size two in the statement and proof of Theorem~\ref{lem:main}, that one can prove an analogous formula for cliques. One can then establish an analogous version of Corollary~\ref{cor:ithree} and Corollary~\ref{lem:reg}. 
\end{remark}

\section{Proof of Theorem~\ref{thm:mainresultofmindeg}}
\label{sec:indmindeg}

\begin{backthm}[\ref{thm:mainresultofmindeg}]
	For any $n,\d\in\N$, if $G$ is an $n$-vertex graph with $\d(G)\geq\d$, and if $n\geq\frac{(\d+1)(\d+2)}{3}$, then for all $t\geq3$, 
	\[
		i_t(G)\leq i_t(\K).
	\]
\end{backthm}

\begin{proof}
We may assume $\d(G)=\d$ holds, and more strongly, that the removal of any edge from $G$ decreases the minimum degree of $G$; for if this is not the case, we may remove edges until it is, obtaining a graph which has at least as many independent sets of every size as $G$ (since the removal of edges from $G$ cannot decrease $i_t(G)$ for any $t$). This assumption was used throughout \cite{GEN}, where such a graph was called edge-minimum-critical. We shall adopt this term.  For the duration of this proof, all graph parameters are taken to be functions of $G$ unless specified otherwise (e.g., $V:=V(G)$ and $d(v):=d_G(v)$) except those defined directly by (\ref{setsofsubgraphsofgraph}) (e.g., $P_3(G)$, to distinguish the $3$-vertex path $P_3$ from the set of $3$-vertex paths in $G$). 

We prove the result in two cases. First assume that $\D\leq n-\d-1$. We show, by induction on $t$, that in this case, 
\begin{equation}
i_t\leq\binom{n-\d}{t}\label{eq:inductionassumption}
\end{equation}
for all $t\geq 3$. To establish the base case for this induction, we first notice that Corollary~\ref{cor:ithree}(\ref{cor:ithree:a}) immediately provides the bound
\begin{align}
i_3&\leq\binom{n}3-(n-2)m+\sum_{v\in V}\binom{d(v)}2\\
	&=\binom{n}3-\frac{n-2}{2}\sum_{v\in V}d(v)+\sum_{v\in V}\binom{d(v)}2\\
	&=\binom{n}{3}-\sum_{v\in V}\frac{d(v)(n-d(v)-1)}2\\
	&\leq \binom{n}{3}-\sum_{v\in V}\frac{\d(n-\d-1)}2,\\
	&=\binom{n}{3}-\frac{n\d(n-\d-1)}2\label{eq:firstithreeboundinmaxdegcase}
\end{align}
with the last inequality holding as each term of the sum is a parabola in $d(v)$ opening downward; so, to maximize $i_3$, we seek to minimize the sum, and this is done term-wise by taking the most extreme values of $d(v)$. This is achieved when $d(v) = \delta$ or $d(v) = \Delta\leq n- \d -1$. It is straightforward to check that the right side of (\ref{eq:firstithreeboundinmaxdegcase}) is bounded above by $\binom{n-\d}{3}$ if and only if $\d(3n-\d^2-3\d-2)$ is nonnegative, which is equivalent to $n\geq\frac{(\d+1)(\d+2)}3$.

Now that the base case for induction is established, assume, for some $s>3$, that (\ref{eq:inductionassumption}) holds for all $3\leq t<s$. Notice that, since each independent set of $G$ of size $s$ can be thought of as an independent set of size $s-1$, together with a vertex not in that independent set or adjacent to any vertex of it, we have that
\begin{equation}
i_s=\frac{1}{s}\abs{\setof{(I,v)\in I_{s-1}(G)\times V}{v\notin N[I]}},\label{eq:toworkfrom}
\end{equation}
where $N[I]:=\cup_{v\in I} N[v]$ for each $I\in I_{s-1}(G)$. Since, for each $v\in I$, $d(v)\geq\d$, $|N[I]|\geq \d+s-1$ for each $I\in I_{s-1}(G)$. Therefore, we can work from (\ref{eq:toworkfrom}) with our induction assumption to obtain that
\begin{align}
i_s&=\frac{1}{s}\sum_{I\in I_{s-1}(G)}\abs{V\setminus N[I]}\\
	&\leq\frac{n-\d-s+1}{s}i_{s-1}\\
	&\leq \frac{n-\d-s+1}{s}\binom{n-\d}{s-1},
\end{align}
which is exactly $\binom{n-\d}{s}$, as desired. This completes the proof in the case that $\D\leq n-\d-1$. 

Now assume that $\D\geq n-\d$. We handle this case inductively as well, but this time fixing $t\geq3$ and inducting on $\d$. To that end, fix $t\geq 3$. To establish a base case, first assume that $\d=1$. Consider $v\in V$ such that $d(v)\geq n-\d=n-1$, which exists as $\D\geq n-\d$. It must be that $v$ is adjacent to every vertex in $V\setminus\{v\}$. If we remove all edges in $E(G[N(v)])$, we obtain a graph which does not have less independent sets of size $t$ (since the removal of edges cannot decrease $i_t$), and in particular, we obtain $K_{1,n-1}$, and therefore the desired result in this case. 

Now, let $\d>1$ and assume that the result holds for all graphs of minimum degree less than $\d$ which satisfy the given hypotheses. Choose a vertex $v\in V$ satisfying $d(v)\geq n-\d$, furnished by our case assumption. Consider graph $G-v$, the graph defined by $V(G-v)=V\setminus\{v\}$ and $E(G-v)=\{e\in E:v\notin e\}$. We would like to apply our induction assumption to $G-v$, which satisfies $\d(G-v)=\d-1$ by our edge-minimum-critical assumption on $G$, but first must show that $G-v$ satisfies the desired hypotheses of the theorem itself. Since $|V(G-v)|\geq 2(\d-1)$ certainly holds as $n\geq 2\d$ holds, we need only argue that $|V(G-v)|\geq\frac{(\d(G-v)+1)(\d(G-v)+2)}3=\frac{\d(\d+1)}3$. This is the case as our assumption that $n\geq\frac{(\d+1)(\d+2)}3$ implies that $|V(G-v)|=n-1\geq\frac{(\d+1)(\d+2)}2-1=\frac{\d(\d+3)}2$. Thus, by our induction assumption, 
\begin{equation}
i_t(G-v)\leq\binom{n-\d}t+\binom{\d-1}t. \label{eq:almostthere}
\end{equation}

To make use of (\ref{eq:almostthere}), let $i_t(v):=\abs{\setof{I\in I_t(G)}{v\in I}}$, so that
\begin{equation}
i_t=i_t(G-v)+i_t(v). \label{eq:firstofendofmainthm}
\end{equation}
Since the number of independent sets of $I_t(G)$ which contain $v$ is exactly the number of independent sets of size $t-1$ which do not intersect $N[v]$, and since $|N[v]|=d(v)+1=\D+1$ by the way in which $v$ was chosen, we have that 
\begin{equation}
i_t(v)\leq\binom{n-(\D+1)}{t-1}. \label{eq:secondofendofmainthm}
\end{equation}
A combination of (\ref{eq:almostthere}), (\ref{eq:firstofendofmainthm}) and (\ref{eq:secondofendofmainthm}), together with our (maximum degree) case assumption, gives that 
\begin{align}
i_t&\leq \binom{n-\d}t+\binom{\d-1}t+\binom{\d-1}{t-1}\\
	&=\binom{n-\d}t+\binom{\d}t\\
	&=i_t(\K),
\end{align}
as desired, with the first equality holding by a well-known binomial identity. 
\end{proof}

One of the two main results of \cite{GEN} was that Conjecture~\ref{conj:gen} holds when $\d\in\{1,2,3\}$, and the proof of this in \cite{GEN} was quite long. We show that this follows almost immediately from Theorem~\ref{thm:mainresultofmindeg}. 

\begin{corollary}
	For any $n,\d\in\N$, if $G$ is an $n$-vertex graph with $\d(G)\geq\d$, such that $n\geq 2\de$, then for all $t\geq 3$ and $\d\leq 3$,
	\[
		i_t(G)\leq i_t(\K).
	\]
\end{corollary}

\begin{proof}
If $\d=1$, then Theorem~\ref{thm:mainresultofmindeg} tells us that the result holds for all $n\geq 2$, which encompasses all graphs in question (as $n\geq 2\d$). If $\d=2$, then Theorem~\ref{thm:mainresultofmindeg} tells us that the result holds for $n\geq 4$, and the same conclusion is reached. If $\d=3$,  Theorem~\ref{thm:mainresultofmindeg} tells us that the result holds for all $n\geq 7$, and so it remains to show that it holds in this case for $n=2\d$. However, that Conjecture~\ref{conj:gen} holds when $n=2\d$ is not difficult to show, and a short proof can be found in \cite{BIPO}.
\end{proof}

\section{Proofs of Theorems \ref{justificationofgen} and \ref{thm:thelast}} 
\label{sec:bipartite}

\begin{backthm}[\ref{justificationofgen}]
If Conjecture~\ref{conj:our} holds, then so must Theorem~\ref{thm:zzhow}. 
\end{backthm}

\begin{proof}
Assume that $r \geq 1$, and that $G$ is an $n$-vertex, $r$-regular graph satisfying \t{(\ref{eq:generalizedkahn})} for all $0\leq t\leq  n/2$. Then, in particular, 
\begin{equation}
i_t(2rG) \leq [x^t](2(1+x)^r-1)^n. \label{ineqone}
\end{equation}
As is standard, let $P(G;x):=\sum_{t\geq0} i_t(G)x^t$ denote the independence polynomial of a graph $G$. We notice that $i_t(K_{r,r})=[x^t](2(1+x)^r-1)$, and therefore that 
\begin{equation}
i_t(nK_{r,r})=[x^t](2(1+x)^r-1)^n. \label{ineqtwo}
\end{equation}
 From (\ref{ineqone}) and (\ref{ineqtwo}), we have that $P(2rG;x) \leq P(nK_{r,r};x)$ for all $x \geq 0$. Since $P(2rG;x) = P(G;x)^{2r}$ and $P(nK_{r,r};x) = P(K_{r,r};x)^n$, this implies that $P(G;x) \leq P(K_{r,r};x)^\frac{n}{2r}$. Letting $x=1$ yields the desired result. 
\end{proof}

\begin{backthm}[\ref{thm:thelast}]
For any $n,r\in\N$, if $G$ is an $n$-vertex, $r$-regular graph with $r\geq1$, then $G$ satisfies \t{(\ref{eq:generalizedkahn})} for all $t\leq3$, and for $t=4$ when $G$ is assumed triangle-free, with equality holding if and only if $2r$ divides $n$ and $G=\frac{n}{2r}K_{r,r}$. 
\end{backthm}

\begin{proof}
First, assume that $G$ is an $n$-vertex, $r$-regular, triangle-free graph. For the duration of this proof, all graph parameters are taken to be functions of $G$ except those defined directly by (\ref{setsofsubgraphsofgraph}) as in the proof of Theorem~\ref{thm:mainresultofmindeg}. Let $m:=\frac{nr}{2}$. We begin by showing that 
\begin{equation}
i_3(G)=\binom{n}{3}-(n-2)m+n\binom{r}{2}\label{eq:pkjhone}
\end{equation}
and
\begin{equation}
 i_4(G) \leq \binom{n}{4} - \binom{n-2}{2}m+n(n-3)\binom{r}{2}-n\binom{r}{3}+m\brac{\frac{m-r^2}{2}-\frac{(r-1)^2}{4}}.\label{eq:pkjhtwo}
\end{equation}
From Corollary~\ref{lem:reg} and the assumption that $G$ is triangle-free, we immediately see that (\ref{eq:pkjhone}) holds, and that to show (\ref{eq:pkjhtwo}), we need only bound $k_2 \uplus k_2 -c_4$. For each $e\in E$, let $k_2 \uplus k_2(e):=|\setof{J\in K_2\uplus K_2(G)}{e\in E(G[J])}|$ and $c_4(e):=|\setof{J\in C_4(G)}{e\in E(G[J])}|$. Then, since each $J\in K_2\uplus K_2(G)$ contains two edges, and each $J\in C_4(G)$ contains four edges, 
\begin{equation}
k_2\uplus k_2 -c_4= \sum_{e\in E}\frac{k_2 \uplus k_2(e)}{2}-\frac{c_4(e)}{4}.\label{ktwoktwoandcforfourkahn}
\end{equation}
To bound the right side of (\ref{ktwoktwoandcforfourkahn}), consider any $e \in E$, say $e = uv$. Let $x_{N(v)}:= |\{e \in E: e \cap N(v) \neq \emptyset$  and  $e\cap N(u) = \emptyset \}|$. We notice that
\begin{align}
\frac{k_2 \uplus k_2(e)}{2}-\frac{c_4(e)}{4} &=  \frac{m-r^2-x_{N(v)}}{2}-\frac{c_4(e)}{4} \\
&=\frac{m-r^2-x_{N(v)}}{2}-\frac{(r-1)^2-x_{N(v)}}{4}\\
&= \frac{m-r^2}{2}-\frac{(r-1)^2}{4}-\frac{x_{N(v)}}{4},\label{pkjequality}
\end{align}
with the first equality holds as $k_2\uplus k_2(e)$ includes all subgraphs induced by the vertices of $e$ and another edge which do not intersect $N(u)\cup N(v)$, and the second holding by the fact that $G$ is triangle-free (and thus all quadrilaterals containing $e$ must contain an edge between a vertex of $N(u)$ and a vertex of $N(v)$). Since $x_{N(v)}\geq0$, the right side of (\ref{pkjequality}) is upper bounded by $\frac{m-r^2}{2}-\frac{(r-1)^2}{4}$. Thus, it follows from  (\ref{ktwoktwoandcforfourkahn}), that 
\begin{equation}
k_2\uplus k_2 -c_4\leq m\brac{\frac{m-r^2}{2}-\frac{(r-1)^2}{4}},
\end{equation}
and thus, from Corollary~\ref{lem:reg}(\ref{partb}), we have (\ref{eq:pkjhtwo}). Notice that the upper bound in (\ref{eq:pkjhone}) is achieved whenever $G$ is triangle-free, and the upper bound in (\ref{eq:pkjhtwo}) is achieved among triangle-free graphs only when $N(e)$ is a complete bipartite graph for all choices of $e$, or $ G[N[e]] =K_{r,r}$. Thus, this upper bound is achieved exactly when $2r|n$ and $G = \frac{n}{2r}K_{r,r}$.

Now, fix $r \geq 1$. Let $P(n,x)=\sum_{t=0}^4 a_t(n)x^t$ be defined by $a_0(n)=1$, $a_1(n)=n$, $a_2(n)=\binom{n}{2}-m$, $a_3(n)$ the right side of (\ref{eq:pkjhone}) and $a_4(n)$ the right side of (\ref{eq:pkjhtwo}), so that
\begin{equation}
i_t(G)\leq [x^t]P(n,x)\label{eq:eq:eq}
\end{equation}
for all $t\leq 4$. A Taylor expansion of $(2(1+x)^r-1)^\frac{n}{2r}$ about $x=0$ produces a power series, 
\[
(2(1+x)^r-1)^\frac{n}{2r} = \sum_{t\geq0} b_t(n) x^t,
\]
with each $b_t(n)$ a polynomial in $n$. For $t\leq 4$, we have that for any $k\in\N$, $b_t(2kr)=a_t(2kr)$ since these are both counting the number of independent sets of size $t$ in $kK_{r,r}$. Therefore, these polynomials agree in infinitely many values, and thus $b_t(n)=a_t(n)$ for all $n$ and all $t\geq 4$. The result now follows from (\ref{eq:eq:eq}). We note that if the triangle-free assumption is relaxed, then this only strengthens (\ref{eq:pkjhone}), and therefore, that the proof works for $t\leq3$ when the triangle-free condition is not assumed.
\end{proof}

\section{Independent sets of a fixed size in regular graphs} 
\label{sec:regular}

Let us turn our attention back to the following. 

\begin{backcor}[\ref{lem:reg}]
For any $n,r\in\N$, if $G$ is an $n$-vertex, $r$-regular graph, then 
\begin{enumerate}[(a)]
\item $i_3(G)+k_3(G)=\binom{n}{3}-(n-2)\frac{nr}{2}+n\binom{r}{2}$;\label{parta}
\item $i_4(G)-k_4(G)=\binom{n}{4}-\binom{n-2}{2}\frac{nr}2+(n-3)\brac{n\binom{r}{2}-k_3(G)}-n\binom{r}{3}+k_2\uplus k_2(G)-c_4(G)$.\label{partb}
\end{enumerate}
Thus, all $n$-vertex, $r$-regular, triangle-free graphs have the same number of independent sets of size three, and the number of independent sets of size four in a triangle-free regular graph is a function of only the number of quadrilaterals and $K_2\uplus K_2$ subgraphs. 
\end{backcor}

We can see from this, as discussed briefly in Section~\ref{sec:proof}, that the number of independent sets of size $t\in\{3,4\}$ is almost completely a function of the number of cycles of size at most $t$. Thus, if we look at regular graphs which do not have cycles of size four or less, these formulas will give us a great deal of information, and with a bit more restriction, we will be able to determine the number of independent sets of size three and four exactly, or almost exactly. To formalize these ideas, we need some additional definitions. For a graph $G$, we let the girth of $G$, denoted $\girth(G)$, be defined as the length of the shortest cycle in $G$. For two vertices $u,v\in V(G)$, we define the distance between $u$ and $v$, denoted $\t{distance}(u,v)$, to be the length of the shortest $u,v$-path in $G$, i.e., the least number of edges that can be traversed to connect $u$ to $v$. We let the diameter of $G$, denoted $\diam(G)$, be $\max_{u,v\in G}\t{distance}(u,v)$.

In order to demonstrate how the imposition of restrictions on graph parameters such as girth and diameter on the formulas of Corollary~\ref{lem:reg} can give us great amounts of information, we turn our attention to some very particular regular graphs. These are regular graphs of girth $5$ and diameter $2$, which have been a topic of interest among graph theorists since the work of Hoffman and Singleton in \cite{HS}. The study of these graphs is more thoroughly and beautifully motivated in Section $1.5$ of \cite{BRLAMO}. They are part of a larger class of important and well-studied graphs called Moore graphs, which we will not discuss in general, but which can be read about in \cite{BELAO}. It is not difficult, for those familiar, to see that the techniques we are about to use can be extended to many other classes of Moore graphs, for which they may prove convenient.

The celebrated Hoffman-Singleton theorem, an elegant proof of which can be found in the aforementioned section of \cite{BRLAMO}, tells us that if $G$ is an $r$-regular graph satisfying $\girth(G)=5$ and $\diam(G)=2$, then $r\in\{2,3,7,57\}$. Moreover, it has been shown that for $r=2,3$ and $7$, there is a unique such graph in each case: the well-known pentagon, Petersen graph and Hoffman-Singleton graph, respectively. The question of whether such a $57$-regular graph can exist is still open, and of much interest to many. 

\begin{theorem}
For any $n,r\in\N$, if $G$ is an $r$-regular graph with $\girth(G)=5$ and $\diam(G)=2$, then
\begin{enumerate}[(a)]
\item $i_3(G)=\binom{n}{3}-(n-2)m+n\binom{r}{2}$;\label{b}
\item $i_4(G)=\binom{n}{4}-\binom{n-2}{2}m+n(n-4)\binom{r}{2}-n\binom{r}{3}+\binom{m}{2}-m(r-1)^2$;\label{c}
\end{enumerate}
where $n:=|V(G)|$ and $m:=|E(G)|$. In particular, if a $57$-regular graph $G$ of diameter $2$ and girth $5$ exists, then $G$ must have exactly $2^6\cdot3\cdot5^2\cdot7^2\cdot13\cdot19$ pentagons, $2^3\cdot3\cdot5^4\cdot7\cdot11\cdot13\cdot19^2$ independent sets of size three, and $2^2\cdot5^4\cdot7\cdot11\cdot13\cdot19\cdot87751$ independent sets of size four. 
\end{theorem}

\begin{proof}
 For the duration of this proof, all graph parameters are taken to be functions of $G$ except those defined directly by (\ref{setsofsubgraphsofgraph}) as in the proof of Theorem~\ref{thm:mainresultofmindeg}. Since $G$ is $r$-regular and triangle-free, (\ref{b}) follows from Corollary~\ref{lem:reg}(\ref{parta}). To prove (\ref{c}), we begin by showing that
\begin{equation}
c_5=\frac{m(r-1)^2}{5}\label{moorepentagons},
\end{equation}
which will prove useful. To show (\ref{moorepentagons}), consider any $e\in E$, say $e=vw$. Since $G$ is triangle-free, we know that $N(v)\cap N(w)=\emptyset$, and thus that $|N(v)\setminus\{w\}|=|N(w)\setminus\{v\}|=r-1$ since $G$ is $r$-regular. Moreover, for any $x\in N(v)\setminus\{w\}$ and $y\in N(w)\setminus\{v\}$, the number of $J\in C_5(G)$ satisfying $\{v,w,x,y\}\sub J$ is exactly $|N(x)\cap N(y)|$. Since $\girth=5$ and $\diam=2$,$xy\notin E$ and $|N(x)\cap N(y)|=1$. Conversely, every pentagon containing $e$ is obtained this way. Thus, the number of pentagons containing $e$ is exactly $|N(v)\setminus\{w\}\times N(w)\setminus\{v\}|=(r-1)^2$ for any $e\in E$. Since every pentagon contains exactly five edges, this proves (\ref{moorepentagons}). Next, we show that
\begin{equation}
k_2\uplus k_2=\binom{m}{2}-p_3-5c_5\label{ktwoktwoeqtn},
\end{equation}
by first showing that
\begin{equation}
\binom{m}{2}=k_2\uplus k_2+p_3+p_4\label{firstintermidiatesteptoktwoktwo},
\end{equation}
and then showing that
\begin{equation}
p_4=5c_5. \label{relationshipbetweenpfourandcfive}
\end{equation}

To show (\ref{firstintermidiatesteptoktwoktwo}), we prove that $f:E^{(2)}\rightarrow K_2\uplus K_2(G)\cup P_3(G)\cup P_4(G)$ defined by $f(\{e_1,e_2\})=G[e_1\cup e_2]$ is a bijection (noting that $K_2\uplus K_2(G)\cup P_3(G)\cup P_4(G)$ is a disjoint union, but omitting the symbol to avoid confusion). To show injectivity, assume $f(\{e_1,e_2\})=f(\{e_3,e_4\})$ for $\{e_1,e_2\},\{e_3,e_4\}\in E^{(2)}$.  If $|e_1\cap e_2|=1$, then $f(\{e_1,e_2\})\cong P_3$ as $G$ is triangle-free, and clearly $\{e_1,e_2\}=\{e_3,e_4\}$. If $|e_1\cap e_2|=0$, then either $f(\{e_1,e_2\})\cong P_4$ or $f(\{e_1,e_2\})\cong K_2\uplus K_2$ as $\girth=5$. If $f(\{e_1,e_2\})\cong K_2\uplus K_2$, equality is again clear. If $f(\{e_1,e_2\})\cong P_4$, equality is clear if one notes that only the images of the two edges which do not share a vertex in a subgraph isomorphic to $P_4$ can lie in $P_4(G)$. Surjectivity is clear. 

To show (\ref{relationshipbetweenpfourandcfive}), first consider $J\in C_5(G)$, say $J=\{v_1,v_2,\ld, v_5\}$. For any $i\in[5]$, we have  $G[J\setminus\{v_i\}]\cong P_4$. Moreover, no path of length four can lie on more than one pentagon, for this would force a cycle of length three or four.  Thus, $p_4\geq 5c_5$. On the other hand, if $J'=\{w_1,\ld,w_4\}\sub V$ so that $G[J']\cong P_4$, say $w_2\in N(w_1)\cap N(w_3)$ and $w_3\in N(w_2)\cap N(w_4)$, there is $w\in N(w_1)\cap N(w_4)$ as $\diam=2$, and since $G[J']\iso P_4$, $w\notin J'$ by our girth assumption. Furthermore, $J'$ cannot lie on more than one pentagon, again, by our girth assumption. Thus, $p_4\leq 5c_5$, proving (\ref{relationshipbetweenpfourandcfive}), and together with (\ref{firstintermidiatesteptoktwoktwo}), proving (\ref{ktwoktwoeqtn}). Imposing the girth and regularity conditions of $G$ on (\ref{eq:forithree}), which holds for any graph, provides $p_3=n\binom{r}{2}$. A combination of this observation with (\ref{ktwoktwoeqtn}) and (\ref{moorepentagons}) gives
\begin{equation}
k_2\uplus k_2=\binom{m}{2}-n\binom{r}{2}-m(r-1)^2. \label{eq:ktwos}
\end{equation}
The girth assumption of $G$ and (\ref{eq:ktwos}), together with Corollary~\ref{lem:reg}(\ref{partb}), immediately proves (\ref{c}). 
\end{proof}

\section{Closing remarks} 
\label{sec:closing}

In Section~\ref{sec:bipartite}, we  proved some small cases of a generalization of Theorem~\ref{thm:zzhow}, and therefore Theorem~\ref{thm:kahn}, which would be to Theorem~\ref{thm:zzhow} what Conjecture~\ref{conj:gen} is to Conjecture~\ref{conj:galvin}. It would be desirable to have a full proof that these results can, in fact, be extended this way. Also, in section~\ref{sec:regular}, we uncovered some new information about the independent set structure of a $57$-regular graph of girth $5$ and diameter $2$ if such a graph exists. It may be possible  to use these results to better understand the likelihood of its existence. In addition, it seems that it would be possible to use similar techniques to study other classes of graphs with numerous structural restrictions.

\section*{Acknowledgments}

 The first author was financially supported by the UNIDEL Foundation, Inc. This research was carried out under the supervision of University of Delaware professor Felix Lazebnik, whom the authors thank very much for his support and encouragement during the research process, and for his extensive help during the composition process. The authors thank University of Delaware undergraduate Stephen Smith, and writer Laura Alexander, for helping to proofread the first draft of this paper. The authors also thank University of Delaware PhD student Amanda Payne for helping to proofread the final draft of this paper, and Sally Anne Szymanski for copyediting the final draft of this paper. Finally, the authors thank both referees for their many helpful comments which improved this paper very much. 

\np
\bibliographystyle{amsplain}
\bibliography{a-new-method}
\nocite{OTHERCRADO}

\end{document}